\documentclass[12pt,reqno,oneside]{amsart}

\usepackage[T1]{fontenc}
\usepackage[utf8]{inputenc}
\usepackage{lmodern}
\usepackage{microtype}
\usepackage[a4paper,margin=2.8cm,includeheadfoot]{geometry}
\usepackage{amsmath,amssymb,mathtools}
\usepackage{tikz}
\usepackage{lineno}
\usepackage[hidelinks]{hyperref}
\usepackage{aliascnt}
\usepackage[nameinlink,noabbrev]{cleveref}
\usepackage{enumitem}

\allowdisplaybreaks

\newtheorem{theorem}{Theorem}[section]
\newaliascnt{proposition}{theorem}

\aliascntresetthe{proposition}
\newaliascnt{lemma}{theorem}
\newtheorem{lemma}[lemma]{Lemma}
\aliascntresetthe{lemma}
\newaliascnt{corollary}{theorem}

\aliascntresetthe{corollary}
\newtheorem{claim}{Claim}[theorem]
\theoremstyle{definition}
\newaliascnt{definition}{theorem}
\newtheorem{definition}[definition]{Definition}
\aliascntresetthe{definition}
\theoremstyle{remark}
\newaliascnt{remark}{theorem}

\aliascntresetthe{remark}

\crefname{theorem}{theorem}{theorems}
\crefname{proposition}{proposition}{propositions}
\crefname{lemma}{lemma}{lemmas}
\crefname{claim}{claim}{claims}
\crefname{corollary}{corollary}{corollaries}
\crefname{definition}{definition}{definitions}
\crefname{remark}{remark}{remarks}

\newcommand{\cB}{\mathcal B}
\newcommand{\cC}{\mathcal C}
\newcommand{\one}{\mathbf 1}
\newcommand{\VC}{\operatorname{VCdim}}

\newcommand{\addN}{\operatorname{add}(\mathcal N)}

\title[A Non-PAC Consistent Learner in ZFC]
{A Borel Concept Class of VC Dimension One with a Non-PAC Consistent Learner in ZFC}

\author{Mateus Jesus de Arruda Campos}
\address{Instituto de Matem\'atica, Estat\'istica e Ci\^encia da Computa\c{c}\~ao,
Universidade de S\~ao Paulo, Rua do Mat\~ao 1010,
05508-090 S\~ao Paulo, SP, Brazil}
\email{mateusjesus@usp.br}

\author{Gabriel Fernandes}
\address{Instituto de Ci\^encias Matem\'aticas e de Computa\c{c}\~ao,
Universidade de S\~ao Paulo, Avenida Trabalhador S\~ao-carlense 400,
13566-590 S\~ao Carlos, SP, Brazil}
\email{fernandes@icmc.usp.br}

\author{Vinicius de Oliveira Rodrigues}
\address{Instituto de Matem\'atica, Estat\'istica e Ci\^encia da Computa\c{c}\~ao,
Universidade de S\~ao Paulo, Rua do Mat\~ao 1010,
05508-090 S\~ao Paulo, SP, Brazil}
\email{vinior@ime.usp.br}

\subjclass[2020]{Primary 68Q32, 28A05; Secondary 03E17}
\keywords{PAC learning, VC dimension, consistent learning, null ideal, measurability, ZFC}

\date{}

\begin{document}

\begin{abstract}
The fundamental theorem of statistical learning states that, under suitable measurability assumptions, finite Vapnik--Chervonenkis (VC) dimension guarantees that every proper consistent learning rule is probably approximately correct (PAC). Blumer, Ehrenfeucht, Haussler, and Warmuth showed, assuming the Continuum Hypothesis, that the ``well-behavedness'' condition of the concept class cannot be omitted: they constructed a concept class of Borel sets of VC dimension one admitting a consistent learning rule that is not PAC. We show that the Continuum Hypothesis is unnecessary. Working in Zermelo--Fraenkel set theory with the Axiom of Choice (ZFC) alone, we construct a concept class of Borel sets on $[0,1]$ of VC dimension one and a proper consistent learning rule that is not PAC. More precisely, for a suitable Borel probability measure and target concept, the rule has true risk one at every sample size on a set of samples of outer probability one. Consequently, finite VC dimension and Borel measurability of the individual concepts do not suffice to guarantee that every proper consistent learning rule is PAC. The result shows, with no need of extra set-theoretical assumptions, that the additional regularity assumption in the fundamental theorem cannot in general be omitted.
\end{abstract}

\maketitle

\section{Introduction}

The fundamental theorem of statistical learning is a cornerstone in the theory of machine learning.
Under suitable measurability assumptions, it establishes a connection between the
learnability of a concept class and its Vapnik--Chervonenkis (VC) dimension.
Following \cite{BEHW}, under the ``well-behavedness''
condition, finite VC dimension of a Borel class of concepts implies that any consistent learning rule is PAC for the given class.
Building on \cite{DurstDudley}, they show that the hypothesis of well-behavedness cannot be removed if one assumes an extra set-theoretic hypothesis, namely, the Continuum Hypothesis (CH). Under CH, they construct an example of a Borel concept class of VC dimension one and a consistent learning rule that is not PAC.

Their example left open whether the extra set-theoretic
hypothesis is necessary. Pestov \cite{Pestov2011} subsequently studied the same issue, attaining partial results that assumed Martin's Axiom, an additional set-theoretic hypothesis which is weaker than CH and is also independent of ZFC.

In this paper, we settle this question by removing the need for the Continuum Hypothesis or other additional set-theoretic hypotheses: we construct, in the standard Zermelo--Fraenkel set theory with the Axiom of Choice (ZFC), a concept class of Borel sets of VC dimension one and a consistent learning rule that is not PAC.

\begin{theorem}\label{thm:mainDisplay}
  There exists (in ZFC alone) a concept class $\cC$ of Borel sets on the interval $[0, 1]$ with VC dimension one and a consistent learning rule for $\cC$ which is not PAC.
\end{theorem}

More precisely, we construct a Borel probability measure and a Borel class of concepts such that, for a suitable target concept, the rule has true risk one at every sample size on a set of samples of outer probability $1$.

This result shows, without the need of extra set-theoretic assumptions, that the fundamental theorem of statistical learning from \cite{BEHW} cannot be extended to all concept classes of Borel sets merely by dropping the well-behavedness hypothesis.
The theorem demonstrates that the passage from finite VC dimension to consistent PAC learnability is not purely combinatorial.
It delineates a mathematical boundary in the foundations of statistical learning by separating what follows from VC dimension alone from what requires measure-theoretic regularity.

\section{Preliminaries}
\label{sec:preliminaries}

A Polish space is a separable completely metrizable topological space, such as $\mathbb R^n$.
A \emph{concept} is a subset of $X$, and a \emph{concept class} $\cC$ is a collection of concepts.

\begin{definition}[{\cite{VapnikChervonenkis}}]
  Let $X$ be a Polish space and let $\cC$ be a concept class on $X$.
  For a finite set $F\subseteq X$, the trace of $\cC$ on $F$ is $\cC\mathbin{\upharpoonright}F
  =\{H\cap F:H\in\cC\}$.

  A finite set $F$ is \emph{shattered} by $\cC$ if
  $\cC\mathbin{\upharpoonright}F=\mathcal P(F)$.  The \emph{VC dimension} of
  $\cC$ is defined as
  \[
  \VC(\cC)
  =\sup\bigl\{|F|:F\subseteq X\text{ is finite and shattered by }\cC\bigr\}.
  \]

  If the supremum is infinite, we write $\VC(\cC)=\infty$.
\end{definition}

Blumer, Ehrenfeucht, Haussler, and Warmuth
showed, under suitable measurability assumptions, that the VC dimension
characterizes distribution-free binary PAC learnability
\cite[Theorems~2.1(i) and~A2.2]{BEHW}.  Their characterization is the basis
of what is now known as the fundamental
theorem of statistical learning.

If $X$ denotes a topological space, then $\cB(X)$ denotes its
Borel $\sigma$-algebra, that is, the $\sigma$-algebra generated by the open sets of $X$.
A Borel probability measure on $X$ is a probability
measure $\mu$ defined on $\cB(X)$.
If $\mu$ is a Borel probability measure on $X$, we
continue writing $\mu$ for the unique extension to
its completion.
We say that a concept class $\cC$ is Borel if every concept in $\cC$ is a Borel set.

For $m\geq 1$, let $\mu^{\otimes m}$ denote the product probability measure on
$(X^m,\cB(X^m))$.  Assuming that $X$ is Polish,
$\cB(X^m)=\cB(X)^{\otimes m}$.
An element of $X^m$ is called a \emph{sample} of size $m$, or an
\emph{$m$-sample}.
A labeled $m$-sample is an element of $(X\times\{0,1\})^m$.

Given $A\subseteq X$, we denote by $\one_A$ its
\emph{characteristic function}.
For a concept $C\in\cC$ and $s=(x_1,\ldots,x_n)\in X^n$, the $C$-labeled sample
generated by $s$ is
\[
S_C(s)=\bigl((x_1,\one_C(x_1)),\ldots,
              (x_n,\one_C(x_n))\bigr).
\]

We work in the realizable setting: an unknown target concept
$C\in\cC$ labels each $x\in X$ by $\one_C(x)$. Given a finite
$C$-labeled sample, the learner aims to output a hypothesis
$H\in\cC$ with small error relative to $C$.

Formally, a \emph{learning rule for $\cC$}, or a \emph{learner} for $\cC$, is a sequence of maps
\[
L_n:(X\times\{0,1\})^n\longrightarrow\cC \qquad (n\geq1).
\]
Thus, all our learners are proper.
The learner is said to be \emph{consistent} if for every $C\in\cC$, for every $n\geq1$ and every sample $s\in X^n$,
\[
  x_i\in L_n(S_C(s))
  \quad\Longleftrightarrow\quad
  x_i\in C
  \qquad (i=1,\ldots,n).
\]

\begin{definition}[True risk]
\label{def:true-risk}
  Let $X$ be a Polish space, let $\cC$ be a Borel concept
  class on $X$, and let $L=(L_n)_{n\geq1}$ be a learning rule for $\cC$.
  The \emph{true risk} of $L$ with respect to $C\in\cC$, a Borel probability measure $\mu$ on $X$, and $n\geq1$ is the function $r_{C,\mu,n}:X^n\to[0,1]$ defined by
  \[
    r_{C,\mu,n}(s)
      =\mu\bigl(L_n(S_C(s))\mathbin\triangle C\bigr),
  \]
  where $s\in X^n$ and $\mathbin\triangle$ denotes the symmetric difference
  of sets.
  For each $\varepsilon>0$, $n\geq1$, $C\in\cC$, and Borel
  probability measure $\mu$ on $X$, the \emph{bad-sample event} at accuracy
  $\varepsilon$ is
  \[
    B_{C,\mu,n}(\varepsilon)
    =\{s\in X^n:r_{C,\mu,n}(s)>\varepsilon\}.
  \]
\end{definition}
Thus, $r_{C,\mu,n}(s)$ is the probability that
$L_n(S_C(s))$ disagrees with the target $C$ on an independent
point $x\sim\mu$.

\begin{definition}\label{def:pac}
Let $X$ be a Polish space, let $\cC$ be a Borel concept
class on $X$, and let $L=(L_n)_{n\geq1}$ be a learning rule for $\cC$.
The rule $L$ is \emph{probably approximately correct (PAC)} for $\cC$ if, for every
$\varepsilon,\delta>0$, there exists $N\geq 1$ such that, for every $n\geq N$,
every $C\in\cC$, and every Borel probability measure $\mu$ on $X$,
\[
 (\mu^{\otimes n})^*\left(B_{C,\mu,n}(\varepsilon)\right)\leq\delta,
\]
where $(\mu^{\otimes n})^*$ denotes the outer measure induced by $\mu^{\otimes n}$.

We say that $\cC$ is \emph{PAC learnable} if there exists a PAC learning rule for $\cC$.

We say that $\cC$ is \emph{consistently PAC learnable} if every consistent learning rule for $\cC$ is PAC.
\end{definition}

We now state the parts of the fundamental theorem of statistical learning relevant to our discussion.

\begin{theorem}[{\cite[Theorem~2.1]{BEHW}}]
\label{thm:fundamental-behw}
Let $X$ be a Polish space and let $\cC$ be a Borel concept class.
\begin{enumerate}[label=\textup{(\roman*)}]
  \item If $\VC(\cC)=\infty$ and $\cC$ is non-trivial, then $\cC$ is not PAC learnable.
  \item If $\VC(\cC)<\infty$ and $\cC$ is well-behaved, then $\cC$ is consistently PAC learnable.
\end{enumerate}
\end{theorem}

The well-behavedness assumption is a measurability condition on certain
sets of samples; see \cite[p.~953]{BEHW} for its definition. It holds, in
particular, for universally separable Borel classes
\cite[Lemma~A1.1, p.~954]{BEHW}.
It is not needed to prove that if $\VC(\cC)=\infty$ then $\cC$ is not PAC learnable.
However, it is needed in their proof that  $\VC(\cC)<\infty$ implies that $\cC$ is consistently PAC learnable.
In the next section, we show that the well-behavedness assumption cannot be removed.

\section{Proof of {Theorem~\ref{thm:mainDisplay}}}
\label{sec:zfc-counterexample}

The classical CH construction of Durst and Dudley
\cite[Proposition~2.2]{DurstDudley}, as modified by Blumer et al.
\cite[Appendix~A1, p.~953]{BEHW}, uses CH to well-order the whole domain so
that every proper initial segment is countable and hence Borel.  To obtain a
ZFC example with Borel concepts, we replace those initial segments by an
increasing chain of Borel null sets.

Let $I=[0,1]$ with its usual topology, let $\lambda$ be Lebesgue measure on
$I$, and let
\[
 \mathcal N=\{A\subseteq I:\lambda^*(A)=0\}
\]
be its null ideal.  Identifying cardinals with initial ordinals, set
\[
 \addN
 =\min\left\{|\mathcal A|:\mathcal A\subseteq\mathcal N,
                  \ \bigcup\mathcal A\notin\mathcal N\right\}.
\]

The cardinal $\addN$ is known as the additivity of the null ideal.
It is well known that $\aleph_1\leq\addN\leq\mathfrak c$, where $\mathfrak c$
is the cardinality of the continuum.  For background on $\addN$, see
\cite{Blass2010}.

In the construction below, we repeatedly use the fact that if $\mu$ is a Borel probability measure on $I$ and $A\subseteq I$ satisfies $\mu^*(A)=c$, then there exists a Borel set $B\supseteq A$ such that $\mu(B)=c$.

\begin{lemma}
\label{lem:adapted-null-tower}
There exist an increasing family
$\langle C_\alpha:\alpha<\addN\rangle$ of Borel subsets of $I$ and a
Borel probability measure $\mu$ on $I$ such that:
\begin{enumerate}[label=\textup{(\roman*)}]
  \item $\mu(C_\alpha)=0$ for every $\alpha<\addN$, and
  \item $\mu^*\left(\bigcup_{\alpha<\addN}C_\alpha\right)=1$.
\end{enumerate}
\end{lemma}

\begin{proof}
Choose
$\langle A_\xi:\xi<\addN\rangle\subseteq\mathcal N$ such that
$\displaystyle \bigcup_{\xi<\addN}A_\xi\notin\mathcal N$.
Such a family exists by the definition of $\addN$.

Let $\lambda$ be the Lebesgue measure on $I$.
By transfinite recursion, choose a family
$\langle C_\alpha:\alpha<\addN\rangle$ of Borel $\lambda$-null sets such that, for every $\alpha<\addN$,
\[
 A_\alpha\cup\bigcup_{\beta<\alpha}C_\beta\subseteq C_\alpha.
\]
At stage $\alpha$, the set on the left is a union of fewer than $\addN$
members of $\mathcal N$, so it is null by the definition of $\addN$ and has
a Borel null superset.  Thus the recursion is possible and the family
$\langle C_\alpha\rangle$ is increasing.  Put
\[
 Y=\bigcup_{\alpha<\addN}C_\alpha
 \qquad\text{and}\qquad
 c=\lambda^*(Y).
\]
Notice that $c>0$ as
\[
  c=\lambda^*(Y)\geq\lambda^*\left(\bigcup_{\alpha<\addN}A_\alpha\right)>0.
\]

Let $G$ be a Borel set such that $Y\subseteq G$ and $\lambda(G)=c$.
Define a Borel probability measure $\mu$ on $I$ by
\[
 \mu(B)=\frac{\lambda(B\cap G)}{c}
 \qquad(B\in\cB(I)).
\]
Then
\[
  \mu(C_\alpha)=\frac{\lambda(C_\alpha\cap G)}{c}
  =\frac{\lambda(C_\alpha)}{c}=0.
\]
Moreover, if $B\supseteq Y$ is Borel,
then
\[
 Y\subseteq B\cap G\subseteq G,
\]
and hence
\[
 c=\lambda^*(Y)
 \leq\lambda(B\cap G)
 \leq\lambda(G)=c.
\]
Thus $\mu(B)=1$ for every Borel superset $B$ of $Y$, and therefore
$\mu^*(Y)=1$.
\end{proof}

We shall also use the following form of Tonelli's theorem.
For a reference, see \cite[252P]{Fremlin2}.

\begin{theorem}[Tonelli's theorem]
\label{thm:tonelli}
Let $\mu$ be a Borel probability measure on $I$, let $m\geq1$, and let
$f:I\times I^m\to[0,\infty]$ be Borel measurable.
Then the map
\[
 x\longmapsto\int_{I^m}f(x,z)\,d\mu^{\otimes m}(z)
\]
is Borel measurable and
\[
 \int_{I^{m+1}}f\,d\mu^{\otimes(m+1)}
 =\int_I\left(\int_{I^m}f(x,z)\,d\mu^{\otimes m}(z)\right)d\mu(x).
\]
\end{theorem}

The following is well-known, but we include a proof for completeness.

\begin{lemma}\label{lem:full-outer-powers}
Let $\mu$ be a Borel probability measure on $I$, and let $Y\subseteq I$
satisfy $\mu^*(Y)=1$.
For every $n\geq1$,
\[
 (\mu^{\otimes n})^*(Y^n)=1.
\]
\end{lemma}

\begin{proof}
Proceed by induction on $n$.  The assertion for $n=1$ is trivial by hypothesis.
Suppose it holds for $n$, and let
$B\subseteq I^{n+1}$ be a Borel set containing $Y^{n+1}$.  For every
$x\in Y$, the Borel section
\[
 B_x=\{z\in I^n:(x)^\frown z\in B\}
\]
contains $Y^n$, so $\mu^{\otimes n}(B_x)=1$ by the induction hypothesis.  Set
$g(x)=\mu^{\otimes n}(B_x)$.  By \Cref{thm:tonelli}, $g$ is Borel measurable, so
\[
 E=\{x\in I:g(x)=1\}
\]
is Borel, contains $Y$, and therefore has $\mu$-measure one.  Applying
\Cref{thm:tonelli} to $\one_B$ gives
\[
 1\geq\mu^{\otimes(n+1)}(B)
 =\int_I g(x)\,d\mu(x)
 \geq\int_E g(x)\,d\mu(x)
 =\mu(E)=1.
\]
Thus $\mu^{\otimes(n+1)}(B)=1$.  Taking the infimum over Borel supersets $B$
proves the claim.
\end{proof}

Now we are ready to prove the main result of this paper.

\begin{proof}[Proof of \Cref{thm:mainDisplay}]
Fix an increasing family
$\langle C_\alpha:\alpha<\addN\rangle$ of Borel subsets of $I$ and a
Borel probability measure $\mu$ on $I$ as in \Cref{lem:adapted-null-tower}.
Let
\[
 Y=\bigcup_{\alpha<\addN}C_\alpha \qquad \text{and} \qquad \cC=\{\varnothing,I\}
       \cup\{C_\alpha:\alpha<\addN\}.
\]
The class $\cC$ is a Borel concept class that forms a chain containing both $\varnothing$ and
$I$.  Hence it
shatters every singleton and no two-point set, so $\VC(\cC)=1$.

For a labeled sample
$\sigma=((x_1,b_1),\ldots,(x_n,b_n))$, let
\[
 P(\sigma)=\{x_i:b_i=1\}.
\]
For every nonempty finite $P\subseteq Y$, set
\[
 \rho(P)=\min\{\alpha<\addN:P\subseteq C_\alpha\}.
\]
The set on the right is nonempty because $P$ is finite, the family is
increasing, and its union is $Y$.  Define
\[
 L_n(\sigma)=
 \begin{cases}
  \varnothing,&P(\sigma)=\varnothing,\\
  C_{\rho(P(\sigma))},
    &\varnothing\neq P(\sigma)\subseteq Y,\\
  I,&P(\sigma)\nsubseteq Y.
 \end{cases}
\]

\begin{claim}
The learning rule $L=(L_n)_{n\geq1}$ is consistent for $\cC$.
\end{claim}

\begin{proof}[Proof of the claim]
Fix $n\geq1$, a target $C\in\cC$, and a sample $s\in I^n$.  We verify that
$L_n(S_C(s))$ agrees with $C$ on every point of $s$.

If $C=\varnothing$, then every label is zero.  Thus
$P(S_C(s))=\varnothing$, and the learner returns $\varnothing$, which agrees
with all the labels.

Suppose that $C=C_\beta$ for some $\beta<\addN$.  Every positively labeled
sample point belongs to $C_\beta$. If there are no such points, $L_n(S_C(s))=\varnothing$ and hence agrees with the sample.  Otherwise,
\[
 \varnothing\neq P(S_C(s))\subseteq C_\beta\subseteq Y,
\]
so $\rho(P(S_C(s)))\leq\beta$.  Consequently,
\[
 P(S_C(s))\subseteq C_{\rho(P(S_C(s)))}\subseteq C_\beta.
\]
The first inclusion shows that the output contains every positively labeled
sample point.  The second shows that it contains no negatively labeled
sample point, since every such point lies outside the target $C_\beta$.

Finally, suppose that $C=I$.  Then every sample point has label one, so
$P(S_C(s))$ is the set of all points occurring in $s$.  If this set is
contained in $Y$, the learner returns $C_{\rho(P(S_C(s)))}$, which contains
all the sample points.  Otherwise it returns $I$, which also contains all of
them.  Thus the output agrees with every label in this case as well.
\end{proof}

Now we show that the rule is not PAC.
Consider $C=I$ and the Borel probability measure $\mu$.  Every
$C_\alpha$ is $\mu$-null, and the learning rule returns $I$ if and only if at least
one sample point lies outside $Y$.  If $s\in Y^n$, then
$L_n(S_I(s))=C_\alpha$ for some $\alpha<\addN$, and hence
$r_{I,\mu,n}(s)=1$ because $\mu(C_\alpha)=0$.  Otherwise,
$L_n(S_I(s))=I$, and hence $r_{I,\mu,n}(s)=0$.  Therefore,
\[
 r_{I,\mu,n}(s)
 =\mu\bigl(L_n(S_I(s))\mathbin\triangle I\bigr)
 =\one_{Y^n}(s).
\]
For every $0<\varepsilon<1$, it follows that $B_{I,\mu,n}(\varepsilon)=\{s:r_{I,\mu,n}(s)>\varepsilon\}=Y^n$, and
\Cref{lem:full-outer-powers} gives, for every $n\geq1$,
\[
 (\mu^{\otimes n})^*(B_{I,\mu,n}(\varepsilon))=(\mu^{\otimes n})^*(Y^n)=1.
\]
Taking, for instance, $\varepsilon=\delta=1/2$ contradicts Definition~\ref{def:pac}.
Thus, the rule is not PAC, and the proof of \Cref{thm:mainDisplay} is complete.
\end{proof}

\section{Conclusion}

Theorem~\ref{thm:mainDisplay} shows that the passage from finite VC dimension to consistent PAC learnability is not purely combinatorial.  Some measure-theoretic regularity is needed to connect the trace bounds supplied by finite VC dimension to probabilities of the bad-sample events.  The
assumption that each individual concept is Borel does not suffice.
We note that in our construction, we do not guarantee that the class is not PAC learnable: the theorem exhibits a consistent learning rule that fails to be PAC, but does not rule out the existence of another PAC learning rule for the same class.

Earlier counterexamples relied on additional set-theoretic assumptions, most
notably the Continuum Hypothesis, and therefore showed only that such a
failure is consistent with ZFC. Our construction removes this dependence by
producing a counterexample in ZFC alone. Consequently, ZFC refutes the
assertion that every Borel concept class of finite VC dimension is
consistently PAC learnable.

\section*{Acknowledgments}

The authors gratefully acknowledge financial support from the S\~ao Paulo
Research Foundation (FAPESP) through grants 2025/07302-0 (Vinicius de
Oliveira Rodrigues) and 2025/09425-1 (Gabriel Fernandes).

\section*{Statements and declarations}

\noindent\textbf{Competing interests.}
The authors declare no competing interests.

\smallskip
\noindent\textbf{Data availability.}
Data sharing is not applicable to this article because no datasets were
generated or analyzed.

\smallskip
\noindent\textbf{Use of generative artificial intelligence.}
OpenAI's GPT-5.6 Sol, accessed through Codex, was used for brainstorming and
testing ideas, producing first drafts of proofs, and assisting with
bibliographic searches, proofreading, LaTeX editing, and manuscript revision.
All AI-generated output and suggested references were reviewed, hand-edited,
and revised by the authors, who take full responsibility for the final
manuscript.

\bibliographystyle{amsplain}
\bibliography{includes/references}

\end{document}